\documentclass[a4paper,11pt]{article}

\usepackage[T1]{fontenc}
\usepackage[utf8]{inputenc}
\usepackage{graphicx}
\usepackage{amssymb, amsthm, amsmath}
\usepackage{mathtools,hyperref}
\usepackage{bbm}
\usepackage{tikz}

\newtheorem{theorem}{Theorem}[section]

\newtheorem{remark}[theorem]{Remark}
\newtheorem{lemma}[theorem]{Lemma}
\newtheorem*{theorem*}{Theorem}

\DeclareMathOperator{\E}{\mathbb{E}}
\DeclareMathOperator{\p}{\mathbb{P}}
\DeclareMathOperator{\Var}{Var}
\newcommand{\N}{\mathbb{N}}
\newcommand{\1}{\mathbbm{1}}

\title{A universal right tail upper bound 
	for supercritical Galton-Watson processes
	with bounded offspring}

\author{
	John Fernley\footnote{Alfr\'ed R\'enyi Institute of Mathematics, Budapest, Hungary.
		email: {fernley@renyi.hu}}
	\and 
	Emmanuel Jacob\footnote{\'Ecole Normale Sup\'erieure de Lyon,
		France. email: {emmanuel.jacob@ens-lyon.fr}}
}

\begin{document}

\maketitle

\abstract{
We consider a supercritical Galton-Watson process $Z_n$ whose offspring distribution has mean $m>1$ and is bounded by some $d\in \{2,3,\ldots\}$. As is well-known, the associated martingale $W_n=Z_n/m^n$ converges a.s. to some nonnegative random variable $W_\infty$. We provide a universal upper bound for the right tail of $W_\infty$ and $W_n$, which is uniform 
	in $n$ and in all offspring distributions with given $m$ and $d$, namely:
	\[
	\p(W_n\ge x)\le c_1 \exp\left\{-c_2 \frac {m-1}m \frac x d\right\}, \quad \forall n\in \N\cup \{+\infty\},  \forall x\ge 0,
	\]
	for some explicit constants $c_1,c_2>0$.
	For a given offspring distribution, our upper bound decays exponentially as $x\to \infty$, which is actually suboptimal, but our bound is \emph{universal}: it provides a single \emph{effective} expression -- which does not require large $x$ -- and is valid simultaneously for all supercritical bounded offspring distributions.
}


\section{Introduction}

Consider a Galton-Watson process $(Z_n)_{n\ge 0}$ with $Z_0=1$ where the offspring distribution 
has expectation $m=\E[Z_1]>1$ and $d=\sup\{k\in \N, \p(Z_1=k)>0\}<+\infty$. In other words, the GW process is supercritical with bounded offspring distribution. By definition,
\[Z_{n+1}=\sum_{k=1}^{Z_n} \xi_{n,i},\]
where $(\xi_{n,i})$ is a collection of independent identically distributed random variables with the same distribution as $Z_1$. The martingale $W_n=Z_n/m^n$ converges almost surely and in $L^2$ to some random variable $W_\infty\ge 0$, see for example \cite{athreyaney}. 
The seminal work of Harris~\cite{harris48} shows that $W_\infty$ has finite exponential moments and provides the following estimates for the generating function:
\[
\varphi(x):=\log \E e^{xW_\infty} = x^{\gamma} C(x)+ O(1) \quad \text{ as } x\to \infty,
\]
where $C$ is a positive multiplicatively periodic function (in particular it is bounded above and below in $(0,+\infty)$), $\gamma$ is defined by $m^\gamma=d$ or $\gamma=\ln d/\ln m$, and $O(1)$ is the standard Landau notation. Harris' work is based on the Poincaré functional equation satisfied by $\varphi$, namely $\varphi(mx)=f(\varphi(x))$ where $f$ is the generating function of $Z_1$. Biggins and Bingham~\cite{biggins1993large} could use this result to provide asymptotic estimates for the tail of $W_n$ for both cases $n<+\infty$ or $n=+\infty$. This was later refined by Fleischmann and Wachtel in~\cite{wachtel2009LeftTail}, who were mainly studying the left tail of $W_\infty$, but also expressed for the right tail the following asymptotics (see Remark 3 in their paper):
\begin{equation}\label{eq:FW}
\p(W_\infty>x) = (1+o(1)) x^{\frac {-\gamma}{2(\gamma-1)}} C_1(x) e^{- C_2(x) x^{\frac \gamma {\gamma-1}}},
\end{equation}
where $C_1$ and $C_2$ are again positive multiplicatively periodic functions. 
In 2013, Denisov, Korshunov, and Wachtel~\cite{wachtel2013tail} provide similar results when the offspring distribution is no more bounded but on the contrary heavy tailed, relying on probabilistic techniques for sums of i.i.d. random variables rather than on the Poincaré functional equation.  
Their article is also a good reference for an overview on the topic.
\bigskip

In this note, we seek for a more effective upper bound for the tail of $W_n$. We seek for a bound that would be \emph{nonasymptotic}, available for both $n$ finite or infinite, and above all \emph{uniform} over the possible choices of the offspring distribution, or at least only depending on $m$ and $d$. To motivate this, observe that the optimal strategy leading to large values of $W_\infty$ often stems in just asking the first generations to have maximal degree (see~\cite{wachtel2009LeftTail}). For example, if you assume $\eta=\p(Z_1=d)>0$, you have $W_k=x=d^k/m^k$ with probability
\[
\eta^{1+d+\ldots+d^{k-1}}=\eta^{\Theta(d^{k-1})}=\eta^{\Theta\left( {x^{\gamma/(\gamma-1)}} d^{-1}\right)},
\]
which already somehow resembles the asymptotics given in~\eqref{eq:FW}. The maximal value of $\eta$ being of course $m/d$, we obtain a uniform upper bound for the event that $W_k$ equals its maximal possible value $d^k/m^k$. It is natural to ask whether we can obtain a similar uniform upper bound for the event $\{W_\infty>x\}$, possibly an upper bound that would feature the term $\exp(-c x^{\gamma/(\gamma-1)}d^{-1})$, or be as close as possible to~\eqref{eq:FW}.

We didn't find any such result in the literature. In this note, we thus provide, up to our knowledge, a first result in that direction. 

\begin{theorem}\label{thm:tail_GWBall}
	There are universal constants $c_1, c_2>0$ such that for any $m>1$, any $d\in\{2,3,\ldots\}$ and any offspring distibution on $\{0,\ldots, d\}$ with mean $m$, for all  $n\in \N \cup\{+\infty\}$ and $x\ge 0$, we have
	\begin{equation}\label{ineq:RTailUpperbound}
	\p(W_n\ge x)\le c_1 \exp\left\{-c_2 \frac {m-1}m \frac x d\right\}.
	\end{equation}
\end{theorem}

\begin{remark}
	Our original motivation for proving Theorem~\ref{thm:tail_GWBall} lies in the study of the contact process on power-law trees in~\cite{fj23}. In that work the other existing upper bounds were not really useful, and it required a more effective upper bound like~\eqref{ineq:RTailUpperbound}.
\end{remark}

\begin{remark}
	We didn't  seek optimal constants, but in our proof we can take $c_2=(\log 3-1)/40$. Less importantly, we can take $c_1=2$, or $c_1=1$ if we discard small values of $x$.
\end{remark}

\section{Proof of Theorem~\ref{thm:tail_GWBall}}

Heuristically, we obtain the upper bound by arguing that in order to have $\{W_n>xd\}$, we need to have simultaneous occurrence of $\Theta(x)$ unlikely events, corresponding to the existence of unusually fecund individuals. For each individual $x$ of the branching process, we define $|x|$ its generation, so the ancestor has generation 0, its direct children generation 1, and so on. For $n\ge |x|$ we define $Z^x_n$ the number of descendants of $x$ of generation $n$, so $Z^x_{|x|}=1$ and for each individual $x$, we have that $(Z^x_{n+|x|})_{n\ge 0}$ has the same law as $(Z_n)_{n\ge 0}$. We further define
 \[n_x:=\inf_n\{n>|x|, Z^x_{n}> ad m^{n}\},
 \]
 with $a$ some parameter that we will fix later on, and we say $x$ is \emph{fecund} if $n_x$ is finite, or in other words if the descendants of $x$ alone can give a contribution to $W_n$ (for some $n>|x|$) larger than $ad$.
 
 In order to determine that an individual $x$ is fecund, there is no need to reveal the whole progeny of $x$, or event the whole progeny up to generation $n_x$. It suffices to reveal its progeny up to generation $n_x-1$ (which automatically satisfies $Z^x_{n_x-1} \leq a d m^{n_x-1}$ by the definition of $n_x$), and then reveal one by one the offspring of these individuals of generation $n_x-1$, up to a point where we have revealed the presence of at least $adm^{n_x}$ individuals of generation $n_x$. Note that at this point:
 \begin{itemize}
 	\item There is a number $A_x$ of descendants of $x$ of generation $n_x-1$ with still unrevealed offspring, for some $A_x$ satisfying $0\le A_x<Z^x_{n_x-1}\le a d m^{n_x-1}$.
 	\item We have already revealed the presence of a number $B_x$ of descendants of $x$ of generation $n_x$, with
 	\[adm^{n_x}\le B_x<adm^{n_x}+d.
 	\]
 	The upper bound on $B_x$ here is a simple consequence of the fact that the offspring of each individual is bounded by $d$. It does not prevent that we might still have $Z^x_{n_x} \geq adm^{n_x}+d.$
 \end{itemize}
	We now construct recursively a multitype branching tree $T$ as follows\footnote{In order to improve readibility of this note, we try and reserve the terminology of individuals/population and so on to the original branching process, while we reserve the graph/tree terminology to this second multitype branching process.}:
\begin{itemize}
	\item The root of $T$ is the ancestor of the branching process, and has type 0. More generally, each vertex of the tree corresponds to or is actually iden-tified with some individual $x$ of the branching process, and has type $|x|$.
	\item If a vertex of $T$ corresponds to a fecund individual $x$ of the branching process, then in $T$ it has $A_x$ children ot type $n_x-1$ and $B_x$ children of type $n_x$, which correspond of course to the previously described descendants of $x$ in the branching process of generation $n_x-1$ and $n_x$.
	\item Otherwise, the vertex is a leaf of $T$.
\end{itemize}
Note that the internal nodes of $T$ are necessarily fecund individuals of the branching process, and in particular $T$ contains only one vertex if the ancestor is not fecund. 
The next two lemmas allow to bound the number of internal nodes of $T$, and then the size of $Z_n$.

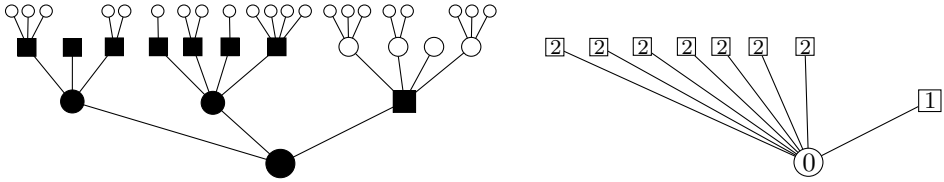
\begin{figure}


\begin{tikzpicture}[x=0.65pt,y=0.65pt,yscale=-1,xscale=1]
	
	\draw  [fill={rgb, 255:red, 0; green, 0; blue, 0 }  ,fill opacity=1 ] (264.24,69.66) -- (277.08,69.66) -- (277.08,82.34) -- (264.24,82.34) -- cycle ;
	\draw  [fill={rgb, 255:red, 0; green, 0; blue, 0 }  ,fill opacity=1 ] (192.08,39.14) -- (202.49,39.14) -- (202.49,49.42) -- (192.08,49.42) -- cycle ;
	\draw  [fill={rgb, 255:red, 0; green, 0; blue, 0 }  ,fill opacity=1 ] (165.1,39.14) -- (175.51,39.14) -- (175.51,49.42) -- (165.1,49.42) -- cycle ;
	\draw  [fill={rgb, 255:red, 0; green, 0; blue, 0 }  ,fill opacity=1 ] (143.68,39.14) -- (154.1,39.14) -- (154.1,49.42) -- (143.68,49.42) -- cycle ;
	\draw    (154.05,23.52) -- (148.89,44.28) ;
	\draw    (213.54,23.52) -- (197.28,44.28) ;
	\draw    (203.23,23.52) -- (197.28,44.28) ;
	\draw    (193.71,23.52) -- (197.28,44.28) ;
	\draw    (183.4,23.52) -- (197.28,45.06) ;
	\draw    (169.91,23.52) -- (170.31,44.28) ;
	\draw    (143.73,23.52) -- (148.89,44.28) ;
	\draw    (148.89,44.28) -- (160.39,76.78) ;
	\draw    (170.31,44.28) -- (160.39,76.78) ;
	\draw  [fill={rgb, 255:red, 0; green, 0; blue, 0 }  ,fill opacity=1 ] (98.47,39.14) -- (108.88,39.14) -- (108.88,49.42) -- (98.47,49.42) -- cycle ;
	\draw    (44.57,23.52) -- (53.69,44.28) ;
	\draw    (54.09,23.52) -- (53.69,44.28) ;
	\draw    (64.4,23.52) -- (53.69,44.28) ;
	\draw    (227.82,23.52) -- (238.53,44.28) ;
	\draw    (238.14,23.52) -- (238.53,44.28) ;
	\draw    (247.66,23.52) -- (238.53,44.28) ;
	\draw    (261.94,23.52) -- (267.09,44.28) ;
	\draw    (272.25,23.52) -- (267.09,43.5) ;
	\draw    (300.01,23.52) -- (309.14,44.28) ;
	\draw    (309.53,23.52) -- (309.14,44.28) ;
	\draw    (319.85,23.52) -- (309.14,44.28) ;
	\draw    (309.14,44.28) -- (270.66,76) ;
	\draw    (238.53,44.28) -- (270.66,76) ;
	\draw    (266.3,43.5) -- (270.66,76) ;
	\draw    (287.72,44.28) -- (270.66,76) ;
	\draw    (128.66,23.52) -- (129.06,44.28) ;
	\draw    (109.62,23.52) -- (103.67,45.06) ;
	\draw    (100.1,23.52) -- (103.67,45.06) ;
	\draw  [color={rgb, 255:red, 0; green, 0; blue, 0 }  ,draw opacity=1 ][fill={rgb, 255:red, 0; green, 0; blue, 0 }  ,fill opacity=1 ] (190.93,112.03) .. controls (190.93,107.48) and (194.66,103.8) .. (199.26,103.8) .. controls (203.86,103.8) and (207.59,107.48) .. (207.59,112.03) .. controls (207.59,116.57) and (203.86,120.25) .. (199.26,120.25) .. controls (194.66,120.25) and (190.93,116.57) .. (190.93,112.03) -- cycle ;
	\draw  [color={rgb, 255:red, 0; green, 0; blue, 0 }  ,draw opacity=1 ][fill={rgb, 255:red, 0; green, 0; blue, 0 }  ,fill opacity=1 ] (72.73,76) .. controls (72.73,72.32) and (75.75,69.34) .. (79.48,69.34) .. controls (83.2,69.34) and (86.22,72.32) .. (86.22,76) .. controls (86.22,79.68) and (83.2,82.66) .. (79.48,82.66) .. controls (75.75,82.66) and (72.73,79.68) .. (72.73,76) -- cycle ;
	\draw  [color={rgb, 255:red, 0; green, 0; blue, 0 }  ,draw opacity=1 ][fill={rgb, 255:red, 0; green, 0; blue, 0 }  ,fill opacity=1 ] (153.65,76.78) .. controls (153.65,73.11) and (156.67,70.13) .. (160.39,70.13) .. controls (164.12,70.13) and (167.14,73.11) .. (167.14,76.78) .. controls (167.14,80.46) and (164.12,83.44) .. (160.39,83.44) .. controls (156.67,83.44) and (153.65,80.46) .. (153.65,76.78) -- cycle ;
	\draw    (53.35,44.62) -- (79.13,76.34) ;
	\draw  [color={rgb, 255:red, 0; green, 0; blue, 0 }  ,draw opacity=1 ][fill={rgb, 255:red, 255; green, 255; blue, 255 }  ,fill opacity=1 ] (232.98,44.28) .. controls (232.98,41.25) and (235.47,38.8) .. (238.53,38.8) .. controls (241.6,38.8) and (244.09,41.25) .. (244.09,44.28) .. controls (244.09,47.31) and (241.6,49.76) .. (238.53,49.76) .. controls (235.47,49.76) and (232.98,47.31) .. (232.98,44.28) -- cycle ;
	\draw  [color={rgb, 255:red, 0; green, 0; blue, 0 }  ,draw opacity=1 ][fill={rgb, 255:red, 255; green, 255; blue, 255 }  ,fill opacity=1 ] (261.54,44.28) .. controls (261.54,41.25) and (264.03,38.8) .. (267.09,38.8) .. controls (270.16,38.8) and (272.65,41.25) .. (272.65,44.28) .. controls (272.65,47.31) and (270.16,49.76) .. (267.09,49.76) .. controls (264.03,49.76) and (261.54,47.31) .. (261.54,44.28) -- cycle ;
	\draw  [color={rgb, 255:red, 0; green, 0; blue, 0 }  ,draw opacity=1 ][fill={rgb, 255:red, 255; green, 255; blue, 255 }  ,fill opacity=1 ] (282.16,44.28) .. controls (282.16,41.25) and (284.65,38.8) .. (287.72,38.8) .. controls (290.78,38.8) and (293.27,41.25) .. (293.27,44.28) .. controls (293.27,47.31) and (290.78,49.76) .. (287.72,49.76) .. controls (284.65,49.76) and (282.16,47.31) .. (282.16,44.28) -- cycle ;
	\draw  [color={rgb, 255:red, 0; green, 0; blue, 0 }  ,draw opacity=1 ][fill={rgb, 255:red, 255; green, 255; blue, 255 }  ,fill opacity=1 ] (303.58,44.28) .. controls (303.58,41.25) and (306.07,38.8) .. (309.14,38.8) .. controls (312.2,38.8) and (314.69,41.25) .. (314.69,44.28) .. controls (314.69,47.31) and (312.2,49.76) .. (309.14,49.76) .. controls (306.07,49.76) and (303.58,47.31) .. (303.58,44.28) -- cycle ;
	\draw  [color={rgb, 255:red, 0; green, 0; blue, 0 }  ,draw opacity=1 ][fill={rgb, 255:red, 255; green, 255; blue, 255 }  ,fill opacity=1 ] (96.53,23.52) .. controls (96.53,21.58) and (98.13,20) .. (100.1,20) .. controls (102.07,20) and (103.67,21.58) .. (103.67,23.52) .. controls (103.67,25.47) and (102.07,27.05) .. (100.1,27.05) .. controls (98.13,27.05) and (96.53,25.47) .. (96.53,23.52) -- cycle ;
	\draw  [color={rgb, 255:red, 0; green, 0; blue, 0 }  ,draw opacity=1 ][fill={rgb, 255:red, 255; green, 255; blue, 255 }  ,fill opacity=1 ] (106.05,23.52) .. controls (106.05,21.58) and (107.65,20) .. (109.62,20) .. controls (111.59,20) and (113.19,21.58) .. (113.19,23.52) .. controls (113.19,25.47) and (111.59,27.05) .. (109.62,27.05) .. controls (107.65,27.05) and (106.05,25.47) .. (106.05,23.52) -- cycle ;
	\draw  [color={rgb, 255:red, 0; green, 0; blue, 0 }  ,draw opacity=1 ][fill={rgb, 255:red, 255; green, 255; blue, 255 }  ,fill opacity=1 ] (125.09,23.52) .. controls (125.09,21.58) and (126.69,20) .. (128.66,20) .. controls (130.63,20) and (132.23,21.58) .. (132.23,23.52) .. controls (132.23,25.47) and (130.63,27.05) .. (128.66,27.05) .. controls (126.69,27.05) and (125.09,25.47) .. (125.09,23.52) -- cycle ;
	\draw  [color={rgb, 255:red, 0; green, 0; blue, 0 }  ,draw opacity=1 ][fill={rgb, 255:red, 255; green, 255; blue, 255 }  ,fill opacity=1 ] (140.16,23.52) .. controls (140.16,21.58) and (141.76,20) .. (143.73,20) .. controls (145.7,20) and (147.3,21.58) .. (147.3,23.52) .. controls (147.3,25.47) and (145.7,27.05) .. (143.73,27.05) .. controls (141.76,27.05) and (140.16,25.47) .. (140.16,23.52) -- cycle ;
	\draw  [color={rgb, 255:red, 0; green, 0; blue, 0 }  ,draw opacity=1 ][fill={rgb, 255:red, 255; green, 255; blue, 255 }  ,fill opacity=1 ] (150.48,23.52) .. controls (150.48,21.58) and (152.07,20) .. (154.05,20) .. controls (156.02,20) and (157.62,21.58) .. (157.62,23.52) .. controls (157.62,25.47) and (156.02,27.05) .. (154.05,27.05) .. controls (152.07,27.05) and (150.48,25.47) .. (150.48,23.52) -- cycle ;
	\draw  [color={rgb, 255:red, 0; green, 0; blue, 0 }  ,draw opacity=1 ][fill={rgb, 255:red, 255; green, 255; blue, 255 }  ,fill opacity=1 ] (166.34,23.52) .. controls (166.34,21.58) and (167.94,20) .. (169.91,20) .. controls (171.88,20) and (173.48,21.58) .. (173.48,23.52) .. controls (173.48,25.47) and (171.88,27.05) .. (169.91,27.05) .. controls (167.94,27.05) and (166.34,25.47) .. (166.34,23.52) -- cycle ;
	\draw  [color={rgb, 255:red, 0; green, 0; blue, 0 }  ,draw opacity=1 ][fill={rgb, 255:red, 255; green, 255; blue, 255 }  ,fill opacity=1 ] (179.83,23.52) .. controls (179.83,21.58) and (181.43,20) .. (183.4,20) .. controls (185.37,20) and (186.97,21.58) .. (186.97,23.52) .. controls (186.97,25.47) and (185.37,27.05) .. (183.4,27.05) .. controls (181.43,27.05) and (179.83,25.47) .. (179.83,23.52) -- cycle ;
	\draw  [color={rgb, 255:red, 0; green, 0; blue, 0 }  ,draw opacity=1 ][fill={rgb, 255:red, 255; green, 255; blue, 255 }  ,fill opacity=1 ] (190.14,23.52) .. controls (190.14,21.58) and (191.74,20) .. (193.71,20) .. controls (195.68,20) and (197.28,21.58) .. (197.28,23.52) .. controls (197.28,25.47) and (195.68,27.05) .. (193.71,27.05) .. controls (191.74,27.05) and (190.14,25.47) .. (190.14,23.52) -- cycle ;
	\draw  [color={rgb, 255:red, 0; green, 0; blue, 0 }  ,draw opacity=1 ][fill={rgb, 255:red, 255; green, 255; blue, 255 }  ,fill opacity=1 ] (199.66,23.52) .. controls (199.66,21.58) and (201.26,20) .. (203.23,20) .. controls (205.2,20) and (206.8,21.58) .. (206.8,23.52) .. controls (206.8,25.47) and (205.2,27.05) .. (203.23,27.05) .. controls (201.26,27.05) and (199.66,25.47) .. (199.66,23.52) -- cycle ;
	\draw  [color={rgb, 255:red, 0; green, 0; blue, 0 }  ,draw opacity=1 ][fill={rgb, 255:red, 255; green, 255; blue, 255 }  ,fill opacity=1 ] (209.97,23.52) .. controls (209.97,21.58) and (211.57,20) .. (213.54,20) .. controls (215.52,20) and (217.11,21.58) .. (217.11,23.52) .. controls (217.11,25.47) and (215.52,27.05) .. (213.54,27.05) .. controls (211.57,27.05) and (209.97,25.47) .. (209.97,23.52) -- cycle ;
	\draw    (129.06,45.06) -- (160.39,76.78) ;
	\draw    (197.28,45.06) -- (160.39,76.78) ;
	\draw    (160.39,76.78) -- (199.26,112.03) ;
	\draw    (270.66,76) -- (199.26,112.03) ;
	\draw    (79.48,76) -- (199.26,112.03) ;
	\draw    (79.52,44.62) -- (79.48,76) ;
	\draw    (103.67,44.28) -- (79.48,76) ;
	\draw  [color={rgb, 255:red, 0; green, 0; blue, 0 }  ,draw opacity=1 ][fill={rgb, 255:red, 255; green, 255; blue, 255 }  ,fill opacity=1 ] (224.25,23.52) .. controls (224.25,21.58) and (225.85,20) .. (227.82,20) .. controls (229.79,20) and (231.39,21.58) .. (231.39,23.52) .. controls (231.39,25.47) and (229.79,27.05) .. (227.82,27.05) .. controls (225.85,27.05) and (224.25,25.47) .. (224.25,23.52) -- cycle ;
	\draw  [color={rgb, 255:red, 0; green, 0; blue, 0 }  ,draw opacity=1 ][fill={rgb, 255:red, 255; green, 255; blue, 255 }  ,fill opacity=1 ] (234.57,23.52) .. controls (234.57,21.58) and (236.16,20) .. (238.14,20) .. controls (240.11,20) and (241.71,21.58) .. (241.71,23.52) .. controls (241.71,25.47) and (240.11,27.05) .. (238.14,27.05) .. controls (236.16,27.05) and (234.57,25.47) .. (234.57,23.52) -- cycle ;
	\draw  [color={rgb, 255:red, 0; green, 0; blue, 0 }  ,draw opacity=1 ][fill={rgb, 255:red, 255; green, 255; blue, 255 }  ,fill opacity=1 ] (244.09,23.52) .. controls (244.09,21.58) and (245.68,20) .. (247.66,20) .. controls (249.63,20) and (251.23,21.58) .. (251.23,23.52) .. controls (251.23,25.47) and (249.63,27.05) .. (247.66,27.05) .. controls (245.68,27.05) and (244.09,25.47) .. (244.09,23.52) -- cycle ;
	\draw  [color={rgb, 255:red, 0; green, 0; blue, 0 }  ,draw opacity=1 ][fill={rgb, 255:red, 255; green, 255; blue, 255 }  ,fill opacity=1 ] (258.37,23.52) .. controls (258.37,21.58) and (259.96,20) .. (261.94,20) .. controls (263.91,20) and (265.51,21.58) .. (265.51,23.52) .. controls (265.51,25.47) and (263.91,27.05) .. (261.94,27.05) .. controls (259.96,27.05) and (258.37,25.47) .. (258.37,23.52) -- cycle ;
	\draw  [color={rgb, 255:red, 0; green, 0; blue, 0 }  ,draw opacity=1 ][fill={rgb, 255:red, 255; green, 255; blue, 255 }  ,fill opacity=1 ] (268.68,23.52) .. controls (268.68,21.58) and (270.28,20) .. (272.25,20) .. controls (274.22,20) and (275.82,21.58) .. (275.82,23.52) .. controls (275.82,25.47) and (274.22,27.05) .. (272.25,27.05) .. controls (270.28,27.05) and (268.68,25.47) .. (268.68,23.52) -- cycle ;
	\draw  [color={rgb, 255:red, 0; green, 0; blue, 0 }  ,draw opacity=1 ][fill={rgb, 255:red, 255; green, 255; blue, 255 }  ,fill opacity=1 ] (296.44,23.52) .. controls (296.44,21.58) and (298.04,20) .. (300.01,20) .. controls (301.99,20) and (303.58,21.58) .. (303.58,23.52) .. controls (303.58,25.47) and (301.99,27.05) .. (300.01,27.05) .. controls (298.04,27.05) and (296.44,25.47) .. (296.44,23.52) -- cycle ;
	\draw  [color={rgb, 255:red, 0; green, 0; blue, 0 }  ,draw opacity=1 ][fill={rgb, 255:red, 255; green, 255; blue, 255 }  ,fill opacity=1 ] (305.96,23.52) .. controls (305.96,21.58) and (307.56,20) .. (309.53,20) .. controls (311.51,20) and (313.1,21.58) .. (313.1,23.52) .. controls (313.1,25.47) and (311.51,27.05) .. (309.53,27.05) .. controls (307.56,27.05) and (305.96,25.47) .. (305.96,23.52) -- cycle ;
	\draw  [color={rgb, 255:red, 0; green, 0; blue, 0 }  ,draw opacity=1 ][fill={rgb, 255:red, 255; green, 255; blue, 255 }  ,fill opacity=1 ] (316.28,23.52) .. controls (316.28,21.58) and (317.88,20) .. (319.85,20) .. controls (321.82,20) and (323.42,21.58) .. (323.42,23.52) .. controls (323.42,25.47) and (321.82,27.05) .. (319.85,27.05) .. controls (317.88,27.05) and (316.28,25.47) .. (316.28,23.52) -- cycle ;
	\draw  [color={rgb, 255:red, 0; green, 0; blue, 0 }  ,draw opacity=1 ][fill={rgb, 255:red, 255; green, 255; blue, 255 }  ,fill opacity=1 ] (41,23.52) .. controls (41,21.58) and (42.6,20) .. (44.57,20) .. controls (46.54,20) and (48.14,21.58) .. (48.14,23.52) .. controls (48.14,25.47) and (46.54,27.05) .. (44.57,27.05) .. controls (42.6,27.05) and (41,25.47) .. (41,23.52) -- cycle ;
	\draw  [color={rgb, 255:red, 0; green, 0; blue, 0 }  ,draw opacity=1 ][fill={rgb, 255:red, 255; green, 255; blue, 255 }  ,fill opacity=1 ] (50.52,23.52) .. controls (50.52,21.58) and (52.12,20) .. (54.09,20) .. controls (56.06,20) and (57.66,21.58) .. (57.66,23.52) .. controls (57.66,25.47) and (56.06,27.05) .. (54.09,27.05) .. controls (52.12,27.05) and (50.52,25.47) .. (50.52,23.52) -- cycle ;
	\draw  [color={rgb, 255:red, 0; green, 0; blue, 0 }  ,draw opacity=1 ][fill={rgb, 255:red, 255; green, 255; blue, 255 }  ,fill opacity=1 ] (60.83,23.52) .. controls (60.83,21.58) and (62.43,20) .. (64.4,20) .. controls (66.37,20) and (67.97,21.58) .. (67.97,23.52) .. controls (67.97,25.47) and (66.37,27.05) .. (64.4,27.05) .. controls (62.43,27.05) and (60.83,25.47) .. (60.83,23.52) -- cycle ;
	\draw  [fill={rgb, 255:red, 0; green, 0; blue, 0 }  ,fill opacity=1 ] (48.14,39.48) -- (58.55,39.48) -- (58.55,49.76) -- (48.14,49.76) -- cycle ;
	\draw  [fill={rgb, 255:red, 0; green, 0; blue, 0 }  ,fill opacity=1 ] (123.85,39.14) -- (134.26,39.14) -- (134.26,49.42) -- (123.85,49.42) -- cycle ;
	\draw  [fill={rgb, 255:red, 0; green, 0; blue, 0 }  ,fill opacity=1 ] (74.32,39.48) -- (84.73,39.48) -- (84.73,49.76) -- (74.32,49.76) -- cycle ;
	
	\draw    (573.14,75.56) -- (503.28,111.3) ;
	\draw    (452.87,44.1) -- (503.28,111.3) ;
	\draw    (474.3,44.1) -- (503.28,111.3) ;
	\draw    (357.26,44.44) -- (503.28,111.3) ;
	\draw    (433.02,44.88) -- (503.28,111.3) ;
	\draw    (501.29,44.88) -- (503.28,111.3) ;
	\draw    (383.01,44.1) -- (503.28,111.3) ;
	\draw    (407.62,44.1) -- (503.28,111.3) ;
	\draw  [fill={rgb, 255:red, 255; green, 255; blue, 255 }  ,fill opacity=1 ] (496.08,39.01) -- (506.5,39.01) -- (506.5,49.2) -- (496.08,49.2) -- cycle ;
	\draw  [fill={rgb, 255:red, 255; green, 255; blue, 255 }  ,fill opacity=1 ] (469.09,39.01) -- (479.51,39.01) -- (479.51,49.2) -- (469.09,49.2) -- cycle ;
	\draw  [fill={rgb, 255:red, 255; green, 255; blue, 255 }  ,fill opacity=1 ] (447.66,39.01) -- (458.08,39.01) -- (458.08,49.2) -- (447.66,49.2) -- cycle ;
	\draw  [fill={rgb, 255:red, 255; green, 255; blue, 255 }  ,fill opacity=1 ] (402.41,39.01) -- (412.83,39.01) -- (412.83,49.2) -- (402.41,49.2) -- cycle ;
	\draw  [fill={rgb, 255:red, 255; green, 255; blue, 255 }  ,fill opacity=1 ] (377.4,49.48) -- (387.7,49.48) -- (387.7,39.39) -- (377.4,39.39) -- cycle ;
	\draw  [fill={rgb, 255:red, 255; green, 255; blue, 255 }  ,fill opacity=1 ] (352.05,39.35) -- (362.47,39.35) -- (362.47,49.54) -- (352.05,49.54) -- cycle ;
	\draw  [fill={rgb, 255:red, 255; green, 255; blue, 255 }  ,fill opacity=1 ] (427.81,39.01) -- (438.23,39.01) -- (438.23,49.2) -- (427.81,49.2) -- cycle ;
	\draw  [fill={rgb, 255:red, 255; green, 255; blue, 255 }  ,fill opacity=1 ] (566.71,69.28) -- (579.56,69.28) -- (579.56,81.85) -- (566.71,81.85) -- cycle ;
	\draw  [color={rgb, 255:red, 0; green, 0; blue, 0 }  ,draw opacity=1 ][fill={rgb, 255:red, 255; green, 255; blue, 255 }  ,fill opacity=1 ] (494.94,111.3) .. controls (494.94,106.79) and (498.67,103.14) .. (503.28,103.14) .. controls (507.88,103.14) and (511.61,106.79) .. (511.61,111.3) .. controls (511.61,115.8) and (507.88,119.45) .. (503.28,119.45) .. controls (498.67,119.45) and (494.94,115.8) .. (494.94,111.3) -- cycle ;

	\draw (357.8,44.2) node  [font=\scriptsize]  {$2$};
	\draw (383,44.2) node  [font=\scriptsize]  {$2$};
	\draw (408.1,44.2) node  [font=\scriptsize]  {$2$};
	\draw (433.4,44.2) node  [font=\scriptsize]  {$2$};
	\draw (453.6,44.2) node  [font=\scriptsize]  {$2$};
	\draw (474.6,44.2) node  [font=\scriptsize]  {$2$};
	\draw (501.5,44.2) node  [font=\scriptsize]  {$2$};
	\draw (503.6,111.5) node  [font=\small]  {$0$};
	\draw (573.4,75.56) node  [font=\footnotesize]  {$1$};

\end{tikzpicture}
\caption{On the left we represented the first three generations of an infinite GW process which has $d=4$ and $m=1.75$. We also choose $a=0.5$, so an individual is fecund if it has at least $\lceil adm\rceil=4$ descendants in the first generation, or $\lceil adm^2\rceil=7$ decendants in the second generation, or $\lceil adm^3\rceil=11$ in the third, etc. Here only the ancestor $ \circ$ is fecund, with $n_\circ=2$. This can be observed by revealing only the individuals drawn in black, with in particular $B_\circ=7$ individuals of generation 2 and $A_\circ=1$ individual of generation 1 with unrevealed offspring, all represented by squares. On the right, the corresponding multitype tree $T$, with $\circ$ as single internal vertex, as the other vertices with type 1 and 2 correspond to non-fecund individuals of the GW process.
}
\end{figure}


\begin{lemma}\label{lem:InternalNodes}
	Set $a=\frac {8m}{m-1}$. Then, for any positive integer $r$,
	the probability that $T$ contains at least $r$ internal nodes is at most $e^{-(\log 3-1) r}$.
\end{lemma}

\begin{lemma} \label{lem:Bernstein}
	With $a=\frac {8m}{m-1}$, if you condition on $T$ containing less than $r$ internal nodes, then the probability of $\{W_n>5 a dr\}$ is bounded by $e^{- \frac 2 3 r}$.
\end{lemma}

With these two lemmas at hand, we obtain for any positive integer $r$ the inequality
\[\p(W_n>5a d r)\le e^{-(\log 3-1) r} + e^{- \frac 2 3 r},\]
which translates into Theorem~\ref{thm:tail_GWBall} when taking $x$ nonnegative real number.

\begin{proof}[Proof of Lemma~\ref{lem:InternalNodes}]
	Recall that the martingale $W_n=Z_n/m^n$ converges almost surely and in $L^2$ to $W_\infty$ with $\E[W_\infty]=1$ and $\Var W_\infty=\frac {\Var Z_1}{m(m-1)}$. Stopping the submartingale $W_n^2$ at the first time it hits $[x,+\infty)$, we then obtain a bound for the tail distribution of $\sup_l W_l$ as
	\begin{align*}
	\p(\sup_l W_l\ge x)&\le x^{-2}\E[W_\infty^2 \1_{\{\sup W_l\ge x\}}]\le \left(1+\frac {\Var Z_1}{m(m-1)}\right) x^{-2}\\
	&\le \left(1+\frac {d}{(m-1)}\right) x^{-2},
	\end{align*}
	where in the last inequality we use $\Var Z_1\le \E[Z_1]d=m d$.
	
	Thus, a vertex of $T$ of type $n$ is an internal node with probability 
	\[p_n:=\p\left(\sup_l W_l\ge ad m^n \right)\le 
	\frac 1 {a^2 m^{2n} d^{2n}}+ \frac 1 {a^2 m^{2n} d (m-1)}.
	\]
	For any $n$, we have 
	\[adm^n p_n\le \frac 1 {a}\left(\frac 1 d + \frac 1 {m-1}\right)\le \frac m {a (m-1)}\le \frac 1 8,
	\]
	as well as $p_n\le p_0\le 1/16$.
	We now look at the subtree of internal nodes of $T$. Note that if an internal node has in $T$ (at most) $ad m^{n+1}+d$ children of type $n+1$ and $adm^n$ of type $n$, amongst these children, the number of internal nodes is bounded by
	\begin{align*}
	&\hspace{5mm} \operatorname{Bin}(adm^{n+1}+d,p_{n+1}) \oplus \operatorname{Bin}(ad m^n,p_{n})\\
	&\preceq\  \operatorname{Pois}(2ad m^{n+1}\frac {p_{n+1}}{1-p_{n+1}})\oplus 
	\operatorname{Pois}(ad m^{n}\frac {p_n}{1-p_n})\\
	& \preceq\  \operatorname{Pois}\left(\frac 3 8. \frac {16}{15}\right)=	\operatorname{Pois}\left(\frac 1 3\right), 
	\end{align*}
	where $\operatorname{Bin}(n,p)$ and $\operatorname{Pois}(\lambda)$ stand here for random variables with binomial (respectively Poisson) distribution, $\oplus$ means we sum independent random variables, and $\preceq$ stands for stochastic domination.
	The probability that the root is an internal node is also bounded by the probability that a $\operatorname{Pois}(1/3)$ is positive. Thus in order to have at least $r$ internal nodes, you need to have a $\operatorname{Pois}(mr)$ larger than $r$ where $m=1/3$, which we bound by the usual Chernoff bound
	\[
	\p(\operatorname{Pois}(mr)\ge r) \le e^{-mr} (em)^{r}\le (em)^{r}= e^{-(\log 3-1) r}.
	\]
\end{proof}

\begin{proof}[Proof of Lemma~\ref{lem:Bernstein}]
	Suppose the tree $T$ is finite, with $|i_T|$ internal nodes and a set of leaves  
	$L_T$. For given $n\ge 0$, we bound the size of $Z_n$ by splitting the individuals with respect to their most recent ancestor belonging to $T$ (this most recent ancestor is always well-defined, as the ancestor of the branching process is in $T$). Note that for any vertex $x$ of $T$, at most $adm^n$ individuals of generation $n$ can share $x$ as their most recent ancestor in $T$. We thus obtain the following upper bound for $Z_n$:	
	\begin{equation}\label{recursive_bound}
	Z_n\le | i_T | ad m^n+ \sum_{\ell\in L_T,|\ell|\le n}  Z^\ell_n,
	\end{equation}
	where we recall the notation $Z^\ell_n$ for the number of descendants of the individual $\ell$ at generation $n$. If we condition on $T$, the processes $Z^\ell$, for $\ell\in L_T$, are independent, and $\big(  Z^\ell_k\big)_{k\ge |\ell|}$ is of course the branching process, initiated at time $k=|\ell|$ with $  Z^{\ell}_{|\ell|}=1$, and conditioned\footnote{Indeed, the conditioning on $T$ and on $\ell$ being a leaf of $T$, translates into the conditioning on the individual $\ell$ not being fecund.} on $\Big\{\forall k\ge |\ell|,   Z^\ell_k\le ad m^k\Big\}$. Its value at time $n$ is stochastically dominated by $Z_{n-\ell} \wedge ad m^n$. Dividing by $m^n$, we get
	\begin{equation}
	W_n\le |i_T| ad + \sum_{\ell\in L_T,|\ell|\le n} W_n^{\ell},
	\end{equation}
	where the random variables $  W_n^{\ell}:= m^{-n}  Z_n^{\ell}$ are independent and stochastically dominated by $m^{-|\ell|}W_{n-|\ell|}\wedge ad$. Hence,
	\begin{align*}
	\E[  W_n^{\ell}]&\le m^{-|\ell|},\\
	\Var   W_n^{\ell}&\le m^{-2|\ell|} \E[W_\infty^2]\le  m^{-2|\ell|}  \left(1+\frac {d}{m-1}\right) \\
	&\le  m^{-|\ell|}  \left(1+\frac {d}{m-1}\right).
	\end{align*}
	Summing over the leaves of $T$, we obtain
	\[
	\sum_{\ell\in L_T} \E[  W_n^{\ell}]\le \sum_{\ell\in L_T}  m^{-|\ell|}\le  3 a d |i_T|,
	\]
	using that the sum of the $m^{-|\ell|}$ over the leaves which are the children of a given internal node, is necessarily bounded by $3ad$. Similarly,
	\[
	\sum_{\ell \in L_T} \Var   W_n^{\ell} \le\left(1+\frac {d}{m-1}\right) 3 a d|i_T|\le \frac m{m-1}3 ad^2|i_T| =\frac 3 8 a^2 d^2 |i_T|.
	\]
	Now we have $\E\left[W_n\ \big|\ T\right]\le 4a d |i_T|$, and Bernstein's inequality states
	\begin{align*}
	\p\left(W_n\ge 4 a d|i_T|+ad x\ \Big|\, T\right)
	&\le \exp\left(-\frac {\frac 1 2 \left(adx\right)^2}{\sum \Var   W_n^{\ell} + \frac 1 3 (ad) (a dx)}\right)\\
	&\le \exp\left(-\frac x{\frac 3 4 \frac {|i_T|}x+\frac 2 3}\right),
	\end{align*}
	and hence
	\[
	\p\left(W_n\ge 5adx\ \Big|\, |i_T|\le x\right)\le e^{-\frac {12}{17} x}\le e^{-\frac 2 3 x}.
	\]
\end{proof}

\section*{Acknowledgements}

\noindent
In realising this work John Fernley was a postdoctoral researcher at ENS Lyon, and is now supported by NKFI grant KKP 137490.

\end{document}